\documentclass[11pt]{amsart}
\usepackage{amsmath,amscd,amssymb}
\usepackage[mathscr]{eucal}
\usepackage{amsmath}
\usepackage{enumerate}
\usepackage{color}
\usepackage{mathtools}

 \newtheorem{thm}{Theorem}[section]
 \newtheorem{cor}[thm]{Corollary}
 
 \newtheorem{prop}[thm]{Proposition}

 \theoremstyle{definition}

  \newtheorem{exas}[thm]{Examples}
	\newtheorem{exa}[thm]{Example}
 \theoremstyle{remark}

\numberwithin{equation}{section}

\def\R{\mathbb{R}}
\def\C{\mathbb{C}}
\def\N{\mathbb{N}}
\def\Z{\mathbb{Z}}
\def\D{\mathbb{D}}
\def\T{\mathbb{T}}

\def\ep{\varepsilon}
\def\l{\lambda}

\def\Re#1{\operatorname{Re}#1}

\def\L{\L}

\def\B{\mathcal{B}}

\def\A{\mathcal{A}}
\def\Ker#1{\operatorname{Ker}#1}

\def\L{\mathcal{L}}

\def\dd{d}

\begin{document}

\title[Some developments around the Katznelson-Tzafriri theorem]{Some developments around the Katznelson-Tzafriri theorem}

\author{Charles Batty}
\address{St. John's College\\
University of Oxford\\
Oxford OX1 3JP, UK
}

\email{charles.batty@sjc.ox.ac.uk}

\author{David Seifert}
\address{School of Mathematics, Statistics and Physics\\
 Herschel Building\\ 
 Newcastle University\\
 Newcastle upon Tyne, NE1 7RU, UK
}

\email{david.seifert@ncl.ac.uk}

\begin{abstract}
This paper is a survey article on developments arising from a theorem proved by Katznelson and Tzafriri in 1986 showing that $\lim_{n\to\infty} \|T^n(I-T)\| =0$ if $T$ is a power-bounded operator on a Banach space and $\sigma(T) \cap \T \subseteq \{1\}$.    Many variations and consequences of the original theorem have been proved subsequently, and we provide an account of this branch of operator theory.   
\end{abstract}



\thanks{Corresponding author:  Charles Batty}

%\thanks{}

\date \today

\maketitle

%\tableofcontents
\section{Introduction}

The origins of the Katznelson-Tzafiri theorem go back to the paper \cite{OrnSuc70} where Ornstein and Sucheston proved the following ``zero-two law''.   If  $T$ is a positive contraction on an $L^1$-space, then either $\|T^n(I-T)\| = 2$ for all $n\ge0$, or $\|T^n(I-T)\|\to0$ as $n\to\infty$.   Over a period of time, other theorems of similar type were proved.   Eventually, Katznelson and Tzafriri realised that there was a general theorem which covers all of these cases (see Corollary \ref{KTcor} in this paper).   Furthermore, they were able to extend the scope of their result to more operators and obtain convergence of  $\|T^nf(T)\|$ to $0$ for appropriate functions $f$ (see Theorem \ref{KTthm} for this general version).

The Katznelson-Tzafriri theorem was taken up by other researchers, notably Allan, O'Farrell and Ransford \cite{AOR87,AR89} who linked the theorem to Tauberian theory, and Esterle, Strouse and Zouakia \cite{ESZ1,ESZ2} who expanded the scope of the theorem in various ways.   In particular, they adapted it to $C_0$-semigroups of operators, which has led to many applications to physical models, such as energy decay of damped waves.     Eventually the objectives changed from the question whether there was convergence to $0$ to questions about rates of convergence.   Initially this was studied for $C_0$-semigroups, where answers could most easily be found, and the estimates became very precise provided that the resolvents of the generator could be accurately estimated.   Meanwhile further variants of the Katznelson-Tzafriri theorem in the discrete case were found, and eventually estimates for the rates of convergence were found by adapting the methods used for the continuous-parameter case.

There have been several substantial survey articles on the continuous-parameter results \cite{Bat94,BGTsum,CST20,CT04,Vu97}, but only a few short short surveys in the discrete case \cite{Lek17,Nev97,Zem07}.   The main aim of this survey article is to give a unified, and fairly comprehensive, account of the Katznelson-Tzafriri theorem and the developments that it has led to.  Among the results that we present is a recent new result (Theorem \ref{BSthm}) proved by the authors in a separate paper \cite{BS22}.    We include proofs of some older results, in some cases giving arguments that are new in the discrete setting but are based on known proofs for $C_0$-semigroups.   We include short summaries of some applications in Section \ref{sect4}, and some remarks about semigroups other than $\Z_+$ in Section \ref{sect5}.

\subsection*{Preliminaries}

The unit circle in $\C$ will be denoted by $\T$, the open unit disc by $\D$, and the open right half-plane by $\C_+$.   The closure of $\D$ in $\C$ is $\overline{\D} := \D \cup \T$.

We shall mainly consider bounded holomorphic functions $f$ on $\D$;  they form a Banach algebra $H^\infty(\D)$.   They have boundary functions $f^b$ on $\T$, given by $f^b(z) := \lim_{r\to1-} f(rz)$ where the limit exists for almost all $z\in\T$.   The disc algebra $A(\D)$ is the closed subalgebra of all functions in $H^\infty(\D)$ which have continuous extensions to $\overline{\D}$.    We may write $f$ to mean the function on $\D$ or on $\overline{\D}$, and $f^b$ for the boundary function on $\T$.

Let $(a_n)_{n\in\Z} \in \ell^1(\Z)$, and consider the function $f(z) := \sum_{n=-\infty}^\infty a_n z^n$ for $z \in \T$.   The space of all functions of this form is known as the Wiener algebra  $W(\T)$.  It is a Banach algebra under pointwise multiplication and the norm given by
\[
\|f\|_W := \sum_{n=-\infty}^\infty |a_n|.
\]

For $(a_n)_{n\in\Z_+} \in \ell^1(\Z_+)$, we may define $f(z) = \sum_{n=0}^\infty a_n z^n$ for $z \in \overline{\D}$.  We will denote the space of such functions by $W^+(\D)$, with the norm given by $\|f\|_W := \sum_{n=0}^\infty |a_n|$.   Then $W^+(\D)$ is a subalgebra of $A(\D)$, and it may be identified with the closed subalgebra of $W(\T)$ consisting of functions $f$ where $a_n=0$ for $n<0$.

Let $R$ be a Banach algebra of functions on $\T$ such that $W(\T) \subseteq R \subseteq C(\T)$, with continuous inclusions, and let $R^+$ be the Banach algebra of all functions $f \in A(\D)$  such that $f^b \in R$, so that $W^+(\D) \subseteq R^+ \subseteq A(\D)$.   Let $E$ be a closed subset of $\T$. 
   A function $f \in R^+$ is {\it of spectral synthesis in $R$ with respect to $E$} if there is a sequence $(g_k)_{k\ge1}$  of functions in $R$ such that $g_k$ vanishes on an open set $U_k$ in $\T$ containing $E$, and $\lim_{k\to\infty} \|g_k-f^b\|_R = 0$.   It is clear that such a function $f$ vanishes on $E$.  The set $E$ is said to be {\it a set of spectral synthesis for $R$} if every function $f \in R^+$ which vanishes on $E$ is of spectral synthesis in $R$ with respect to $E$.

 Let $T$ be a power-bounded operator on a complex Banach space $X$, with spectrum $\sigma(T)$ and resolvent set $\rho(T)$,  and let $\sigma_u(T) := \sigma(T) \cap \T$, the  {\it unitary spectrum of $T$}.    Let $\A$ be a Banach algebra such that $W^+(\D) \subseteq \A \subseteq H^\infty(\D)$ with continuous inclusions.   A {\it bounded $\A$-calculus for $T$} is a bounded algebra homomorphism from $\A$ to the algebra $L(X)$ of bounded operators on $X$, written as $f \mapsto f(T)$, mapping the constant function $1$ to the identity operator on $X$ and the identity function $z$ to $T$.   If a bounded $\A$-calculus exists, then there exists $C$ such that $\|p(T)\| \le C \|p\|_\A$ for all polynomials $p$.   The converse is true if the polynomials are dense in $\A$.
 
 For $f \in W^+(\D)$ with Taylor series $f(z) = \sum_{n=0}^\infty a_n z^n$, let $f(T) := \sum_{n=0}^\infty a_n T^n$.   This defines a contractive $W^+(\D)$-functional calculus for $T$. 
 
An operator $T \in L(X)$ is {\it polynomially bounded} if there is a constant $C$ such that
\[
\|p(T)\| \le C \|p\|_\infty
\] for all polynomials $p$, where $\|\cdot\|_\infty$ denotes the supremum norm on $C(\overline\D)$.   There is a bounded $A(\D)$-calculus for $T$ if and only if $T$ is polynomially bounded.
 
 \section{The general case} \label{sect2}

\subsection{The main theorem}
Katznelson and Tzafriri proved the following result in \cite[Theorem 5]{KT86}.

\begin{thm} \label{KTthm}
Let $T$ be a power-bounded operator on a Banach space $X$, and let $f \in W^+(\D)$ be of spectral synthesis in $W(\T)$ with respect to $\sigma_u(T)$.   Then $\lim_{n\to\infty} \|T^n f(T)\| = 0$.
\end{thm}

They made the following observations about the assumptions in their theorem:
\begin{enumerate} [\rm(i)]
\item In order for the theorem to be non-trivial, $\sigma_u(T)$ must be contained in the zero set of $f$ in $\T$ for some non-zero $f \in W^+(\D)$.   This implies that $\sigma_u(T)$ must be of measure zero with respect to Lebesgue measure on $\T$.
\item If $E$ is a countable closed subset of $\T$, then $E$ is a set of spectral synthesis for $W(\T)$.
\item There exist closed subsets $E$ of $\T$ which have measure zero and are not of spectral synthesis for $W(\T)$.
\item Let $f \in W^+(\D)$ vanish on a closed subset $E$ of $\T$.   Then $f$ is of spectral synthesis in $W(\T)$ with respect to $E$ if and only if  $\lim_{n\to\infty} \|T^nf(T)\| = 0$ whenever $T$ is a contraction on any Banach space with $\sigma_u(T) \subseteq E$.   [This observation was attributed to J. Bourgain.]
\end{enumerate}

The proof of Theorem \ref{KTthm} in \cite{KT86} used methods from harmonic analysis.   V\~u \cite{Vu92a} gave a very short proof based on a functional analytic construction and an application of spectral theory for an invertible isometry.  An analogous result for $C_0$-semigroups was proved using similar techniques (see Section \ref{sect5.1}).   In a survey article, Chill and Tomilov \cite[Theorem 5.2]{CT04} gave two more proofs of the continuous-parameter version of Theorem \ref{KTthm}, one via Fourier analysis and one via complete bounded trajectories.   Here we give similar proofs for the discrete version.      

\begin{proof}[Our first proof of Theorem~\ref{KTthm}]
Let $g\in W(\T)$ by given by $g(z)=\sum_{n\in\Z}a_nz^n$ for $z\in\T$, where $\sum_{n\in\Z}|a_n|<\infty$.
Then
$$a_n=\frac{1}{2\pi}\int_{-\pi}^{\pi}e^{-in\theta}g(e^{i\theta})\,\dd\theta,\qquad n\in\Z.$$
Suppose that $g$ vanishes on a neighbourhood of $\sigma_u(T)$, and consider the family $(T^n)_{n\in\Z}$, where we interpret $T^n$ as the zero operator when $n<0$. Using the dominated convergence theorem and Fubini's theorem we see that
$$\begin{aligned}
\sum_{k\in \Z}a_kT^{n+k}&=\lim_{r\to1-}\sum_{k\in\Z} a_kr^{n+k+1}T^{n+k}\\
&=\lim_{r\to1-}\frac{1}{2\pi}\int_{-\pi}^{\pi}g(e^{i\theta})\sum_{k\in\Z} r^{n+k+1}e^{-i k\theta}T^{n+k}\,\dd\theta\\&=\lim_{r\to1-}\frac{1}{2\pi}\int_{-\pi}^{\pi}e^{i(n+1)\theta}g(e^{i\theta})(r^{-1}e^{i\theta}I-T)^{-1}\,\dd\theta\\&=\frac{1}{2\pi}\int_{-\pi}^\pi e^{i(n+1)\theta}g(e^{i\theta})(e^{i\theta}I-T)^{-1}\,\dd\theta
\end{aligned}$$
for all $n\ge0$.  By the Riemann-Lebesgue lemma,
$$\bigg\|\sum_{k\in \Z}a_kT^{n+k}\bigg\|\to0,\qquad n\to\infty. $$

Now let $\ep>0$ and let $g_\ep\in W(\T)$ be a function that vanishes on a neighbourhood of $\sigma_u(T)$ and satisfies $\|f-g_\ep\|_W<\ep$. If   $g_\ep(z)=\sum_{n\in\Z}a_n^{(\ep)} z^n$ for $z\in\T$, then 
$$\|T^nf(T)\|\le K\|f-g_\ep\|_W+\bigg\|\sum_{k\in\Z}a_k^{(\ep)} T^{k+n}\bigg\|<K\ep+\bigg\|\sum_{k\in\Z}a_k^{(\ep)} T^{k+n}\bigg\| $$
for all $n\ge0$, where $K=\sup_{n\ge0}\|T^n\|$. From the first part it follows that
$$\limsup_{n\to\infty}\|T^nf(T)\|\le K\ep,$$
and since $\ep>0$ was arbitrary this implies the result.
\end{proof}

Our second proof of Theorem~\ref{KTthm} is longer but it may be instructive in its own way. It is a slight variation of an argument given in \cite[Section~5.1]{CT04} in the setting of $C_0$-semigroups.   If $T$ is a bounded linear operator on a Banach space $X$ then a sequence $(x^*_n)_{n\in\Z}$ in the dual space $X^*$ of $X$ is a \emph{complete trajectory} for $T^*$ (through $x_0^*$) if $T^*x_n^*=x_{n+1}^*$ for all $n\in\Z$. If, in addition, $(x^*_n)_{n\in\Z}\in\ell^\infty(\Z;X^*)$ then $(x^*_n)_{n\in\Z}$ is a \emph{bounded complete trajectory} for $T^*$. Bounded complete trajectories play an important role in the theory of asymptotic stability, as the following result shows.

\begin{prop}\label{prp:ann}
Let $T$ be a power-bounded operator on a Banach space $X$, and let
$$X_0=\big\{x\in X:\|T^nx\|\to0\mbox{ as }n\to\infty\big\}.$$
  Let $M$ be the set of those $x^*\in X^*$ such that there exists a bounded complete trajectory for $T^*$ through $x^*$, and $M_\perp$ be the annihilator of $M$ in $X$. Then  $M_\perp=X_0$.
\end{prop}

A proof of a more general statement may be found in~\cite[Theorem~3.1]{BaBrGr96}. The less straightforward (and, for us, more important) inclusion relies on the existence of an invertible \emph{limit isometry}; see for instance~\cite[Proposition~2.1]{Vu97}.

Given a sequence $\mathbf{x}^*=(x^*_n)_{n\in\Z}\in\ell^\infty(\Z;X^*)$ we associate with $\mathbf{x}^*$ a function $h_{\mathbf{x}^*}\colon\C\setminus\T\to X^*$ defined by 
$$h_{\mathbf{x}^*}(z)=\begin{cases}
\sum_{n=0}^\infty z^{-n-1}x_{n}^*,&\quad|z|>1,\\
-\sum_{n=1}^\infty z^{n-1}x_{-n}^*,&\quad|z|<1.
\end{cases}
$$
Note that $h_{\mathbf{x}^*}$ is holomorphic, and that 
\begin{equation} \label{hxest}
\|h_{\mathbf{x}^*}(z)\|\le\frac{\|\mathbf{x}^*\|_\infty}{|1-|z||},\qquad|z|\ne1.
\end{equation}
The mapping $\mathbf{x}^*\mapsto h_{\mathbf{x}^*}$ may be regarded as a discrete analogue of the Carleman transform.

\begin{proof}[Our second proof of Theorem~\ref{KTthm}]
We first show that $\|T^nf(T)x\|\to0$ as $n\to\infty$ for all $x\in X$, or equivalently that $\operatorname{Ran} f(T)\subseteq X_0$.   By Proposition~\ref{prp:ann} it suffices to prove that $x^*\in\Ker f(T)^*$ whenever there exists a bounded complete trajectory for $T^*$ through $x^*$. Let $\mathbf{x}^*=(x_n^*)_{n\in\Z}\in\ell^\infty(\Z;X^*)$ be  a bounded complete trajectory for $T^*$, and let $g\in W(\T)$ be given by $g(z)=\sum_{n\in\Z}a_nz^n$ for $z\in\T$, where $\sum_{n\in\Z}|a_n|<\infty$. Then 
$$\bigg\|f(T)^*x_0^*-\sum_{n\in\Z}a_nx_n^*\bigg\|\le\|f-g\|_W\|\mathbf{x}^*\|_\infty,$$
so by the spectral synthesis assumption it suffices to show that $\sum_{n\in\Z}a_nx_n^*=0$ whenever  $g$ vanishes on a neighbourhood of $\sigma_u(T)$. We observe that
$$\begin{aligned}
\sum_{n\in\Z}a_nx_n^*&=\frac{1}{2\pi}\lim_{r\to1-}\sum_{n\in\Z}\int_{-\pi}^\pi r^{|n+1|}e^{-in\theta}g(e^{i\theta})x_n^*\,\dd\theta\\
&=\frac{1}{2\pi}\lim_{r\to1-}\int_{-\pi}^\pi g(e^{i\theta})\bigg(\sum_{n=0}^\infty r^{n+1}e^{-in\theta}x_n^*+\sum_{n=1}^\infty r^{n-1}e^{in\theta}x_{-n}^*\bigg)\,\dd\theta\\
&=\frac{1}{2\pi}\lim_{r\to1-}\int_{-\pi}^\pi e^{i\theta}g(e^{i\theta})\big( h_{\mathbf{x}^*}(r^{-1}e^{i\theta})-h_{\mathbf{x}^*}(re^{i\theta})\big)\,\dd\theta.
\end{aligned}$$
In order to prove our claim we now show that the map $h_{\mathbf{x}^*}$ extends continuously to a neighbourhood of any point in $\rho(T)\cap\T$. We note first that for $w,z\in\C$ with $|w|<1$, $|z|>1$ we have
$$\begin{aligned}
(z-w)(zI-T^*)^{-1}&h_{\mathbf{x}^*}(w)=(w-z)\sum_{n=1}^\infty \sum_{k=0}^\infty w^{n-1}z^{-k-1}\mathbf{x}^*_{k-n}\\
&=(w-z)\sum_{n=1}^\infty w^{n-1}z^{-n}\bigg( \sum_{k=1}^{n} z^{k-1}x^*_{-k}+\sum_{k=0}^\infty z^{-k-1}x_k^*\bigg)\\
&=(w-z)\sum_{k=1}^\infty \sum_{n=k}^\infty w^{n-1}z^{k-n-1}x_{-k}^*-h_{\mathbf{x}^*}(z)\\
&=h_{\mathbf{x}^*}(w)-h_{\mathbf{x}^*}(z).
\end{aligned}$$
Using this identity twice and the fact that $h_{\mathbf{x}^*}(z)=(zI-T^*)^{-1}x_0^*$ gives
$$h_{\mathbf{x}^*}(w)-h_{\mathbf{x}^*}(z)=(z-w)^2(zI-T^*)^{-2}h_{\mathbf{x}^*}(w)+(z-w)(zI-T^*)^{-2}x_0^*$$
for $|w|<1$, $|z|>1$.
Let $\lambda\in\rho(T)\cap\T$.   Clearly $h_{\mathbf{x}^*}(z) \to (\l I - T^*)^{-1}x_0^*$ as $z \to \l,\,|z|>1$.  
 For $w \in \D$, choose $z_w $ such that $|z_w| > 1$ and $|z_w - w| \le 2(1-|w|)$, and let $w \to \lambda$.   
 Then $z_w \to \l$ and hence $h_{\mathbf{x}^*}(z_w)  \to (\l I - T^*)^{-1}x_0^*$.   Using \eqref{hxest}, we obtain
\[
 \|h_{\mathbf{x}^*}(w) -h_{\mathbf{x}^*}(z_w)\| 
\le 5 (1-|w|) \|\mathbf{x}^*\|_\infty \, \|(z_w I - T^*)^{-2}\| \to 0.
\] 
It follows that $h_{\mathbf{x}^*}(w) \to  (\l I - T^*)^{-1}x_0^*$.   Thus
the map $h_{\mathbf{x}^*}$ extends continuously to $\C \setminus \sigma_u(T)$. In fact, by the edge-of-the-wedge theorem \cite[Chapter~1]{Rud71} (or, more directly, by an application of Morera's theorem) $h_{\mathbf{x}^*}$ even extends holomorphically to $\C \setminus \sigma_u(T)$. By the dominated convergence theorem we have
$$\sum_{n\in\Z}a_nx_n^*=\frac{1}{2\pi}\lim_{r\to1-}\int_{-\pi}^\pi e^{i\theta}g(e^{i\theta})\big( h_{\mathbf{x}^*}(r^{-1}e^{i\theta})-h_{\mathbf{x}^*}(re^{i\theta})\big)\,\dd\theta=0$$
 whenever  $g\in W(\T)$ is the function given by $g(z)=\sum_{n\in\Z}a_nz^n$ for $z\in\T$ and $g$ vanishes on a neighbourhood of $\sigma_u(T)$, and we conclude that $\|T^nf(T)x\|\to0$ as $n\to\infty$ for all $x\in X$.

In order to convert this statement into the corresponding statement for the operator-norm topology, we may either consider the operator of left multiplication by $T$ on the Banach algebra $L(X)$, or we may alternatively proceed as follows. Suppose, for the sake of a contradiction, that $\|T^nf(T)\|\not\to0$ as $n\to\infty$. Then we may find a positive number $\ep$, an increasing sequence $(n_k)_{k\ge1}$ in $\Z_+$ and unit vectors $x_k\in X$ such that $\|T^{n_k}f(T)x_k\|\ge\ep$ for all $k\ge1$. Let $T_\infty$ denote the bounded linear operator on $\ell^\infty(\N;X)$ induced by $T$, so that $T_\infty\mathbf{y}=(Ty_k)_{k\ge1}$ for $\mathbf{y}=(y_k)_{k\ge1}\in\ell^\infty(\N;X)$. Then $T_\infty$ is power-bounded and $\sigma(T_\infty)=\sigma(T)$. It follows from what we proved above that $\|T_\infty^nf(T_\infty)\mathbf{y}\|_\infty\to0$ as $n\to\infty$ for all $\mathbf{y}\in\ell^\infty(\N;X)$. However, if we let $\mathbf{x}=(x_k)_{k\ge1}$, then $\mathbf{x}\in\ell^\infty(\N;X)$ and 
$$\|T_\infty^{n_k}f(T_\infty)\mathbf{x}\|_\infty\ge\|T^{n_k}f(T)x_k\|\ge\ep,\qquad k\ge1.$$
This contradiction completes the proof.
\end{proof}

\subsection{Variants}
Subsequently there have been many extensions and variants of the scope of Theorem \ref{KTthm}.    There are several different types of extensions.  One may seek to extend the result to functions $f$ which are not in $W^+(\D)$, but for which $f(T)$ can be defined, or to find conditions on $X$ and $T$ which allow us to replace the assumption of spectral synthesis by the weaker assumption that $f$ vanishes on $\sigma_u(T)$.  We will set out some extensions of these types in the remainder of this section.

 In \cite{AOR87}, Allan, O'Farrell and Ransford showed that if $f' \in W^+(\D)$ and $f$ vanishes on $\sigma_u(T)$ then $\lim_{n\to\infty} \|T^nf(T)\| = 0$.   Moreover, these properties are equivalent to $\sigma_u(T)$ having measure zero and $\sum_{n=0}^\infty z^n T^n f(T)$ converging to $(I-zT)^{-1} f(T)$ whenever $z^{-1} \in \rho(T) \cap \T$.

 In \cite{AR89}, Allan and Ransford extended the scope of Theorem \ref{KTthm} to the context of Banach algebras instead of operators.   Their method of proof involved a functional analytic construction in the context of Banach algebras, but overlapping with the proof in \cite{Vu92a}.    They obtained sharp upper bounds even in cases where spectral synthesis does not hold.   

The paper \cite{ESZ1} contains an abstract result concerning a Banach algebra $R$ satisfying $W(\T) \subseteq R  \subseteq C(\T)$ with continuous inclusions, and with the functions $z^n, \, n\in\Z$, having $R$-norm equal to $1$ and spanning a dense subspace of $R$.  The functions $f \in A(\D)$ such that $f^b \in R$, with the norm $\|f\|_{R^+} := \|f^b\|_R$, form a Banach algebra satisfying $W^+(\D) \subseteq R^+ \subseteq A(\D)$.   Assume that the operator $T$ satisfies
\[
\|p(T)\| \le \|p\|_R,
\]
for all polynomials $p$, so $T$ has a contractive $R^+$-calculus.     If $f\in R^+$ and $f$ is of spectral synthesis in $R$ with respect to $\sigma_u(T)$, then 
\[
\lim_{n\to\infty} \|T^nf(T)\| = 0.
\]
Theorem \ref{KTthm} is essentially the special case of this result when $R = W(\T)$ and $R^+ = W^+(\D)$.    
 At the other extreme where $R = C(\T)$ and $R^+ = A(\D)$, the assumption of a contractive $R^+$-calculus is stronger than polynomial boundedness, which is considerably stronger than power-boundedness.     

The following result for polynomially bounded operators was proved in \cite[Proposition 1.6]{KvN97}.  Spectral synthesis for $C(\T)$ is not needed here because every closed subset of $\T$ is a set of spectral synthesis for $C(\T)$.    The case when
\begin{equation} \label{vnin}
\|p(T)\| \le \|p\|_\infty
\end{equation}
 for all polynomials $p$ was covered by \cite[Theorem 2.10]{ESZ1}.

\begin{thm} \label{KvNthm}
Let $T$ be a polynomially bounded operator on a Banach space.   Let $f \in A(\mathbb D)$ and assume that $f$ vanishes on $\sigma_u(T)$.  Then 
\[
\lim_{n\to\infty} \|T^n f(T)\| = 0.
\]
\end{thm}

Polynomial boundedness is quite common on Hilbert spaces.    If $T$ is a contraction on a Hilbert space $X$, then von Neumann's inequality shows that \eqref{vnin} holds, and Theorem \ref{KvNthm} is applicable.    There are other variants of Theorem \ref{KTthm} for operators on Hilbert spaces as follows. 

The first result, due to Bercovici \cite{Ber90}, concerns completely non-unitary contractions on Hilbert spaces.  Such operators have a bounded $H^\infty(\D)$-calculus defined in the Sz.-Nagy-Foias functional calculus.

\begin{thm}  \label{Berthm}
Let $T$ be a completely non-unitary contraction on a Hilbert space.  Let $f \in H^\infty(\D)$ and assume that $f$ vanishes on $\sigma_u(T)$.   Then $\lim_{n\to\infty} \|T^nf(T)\| = 0$.
\end{thm}

The next result is due to L\'eka \cite[Theorem 2.1]{Lek09}, and it was obtained by using ergodic properties.

\begin{thm}  \label{Lekthm}
Let $T$ be a power-bounded operator on a Hilbert space.  Let $f \in W^+(\D)$ and assume that $f$ vanishes on $\sigma_u(T)$.   Then $\lim_{n\to\infty} \|T^nf(T)\| = 0$.
\end{thm}

The third result is due independently to Mustafayev \cite{Mus10} and Zarrabi \cite[Theorem 3.2]{Zar13}.

\begin{thm} \label{Zarthm}
Let $T$ be a contraction on a Hilbert space with spectral radius $1$, and let $f \in A(\D)$.   Then $\lim_{n\to\infty} \|T^n f(T)\| = \max \{|f(\l)| : \l \in \sigma_u(T)\}$.
\end{thm}

\subsection{Analytic Besov functions}
A new result concerning so-called ``analytic Besov functions'' has been obtained very recently in \cite{BS22}.  Let $\B(\D)$ be the space of  holomorphic functions on $\D$ such that
\begin{equation} \label{bddef}
\|f\|_{\B_0} := \int_0^1 \sup_{|\theta|\le\pi} |f'(re^{i\theta})| \,dr < \infty.
\end{equation}
These functions are uniformly continuous on $\D$, so they extend to functions in $A(\D)$. 
Then $\B(\D)$ is a Banach algebra in the norm
\[
\|f\|_\B := \|f\|_\infty + \|f\|_{\B_0}, \qquad  f \in \B(\D).
\]
The functions $z^k, k\in\Z_+$, belong to $\B(\D)$, with $\B$-norm equal to $2$ for $k\ge1$.   Thus there are continuous inclusions of $W^+(\D)$ in $\B(\D)$ and of $\B(\D)$ in $A(\D)$.   The inclusions are proper, and the polynomials are dense in $\B(\D)$.   

An alternative definition is that $\B(\D)$ is the space of all functions $f \in H^\infty(\D)$ whose boundary functions $f^b$ belong to a certain Besov space on $\T$, usually denoted by $B^0_{\infty1}(\T)$.   Then $\|f\|_\B$ is equivalent to $\|f^b\|_{B^0_{\infty1}}$ for any of the standard norms on $B^0_{\infty1}(\T)$.

A function in $\B(\D)$ satisfies the reproducing formula
\begin{equation} \label{rep}
f(z) = f(0) + \frac{2}{\pi} \int_\D \log \left( \frac{1}{|w|} \right) \frac{zf'(w)}{(1-z\overline{w})^2} \,dA(w),
\end{equation}
where $dA$ denotes area measure on $\D$.   

The formula \eqref{rep} can be adapted to create functions in $\B(\D)$ as follows.   Let $g : \D \to \C$ be a continuous function, such that
\[
\int_0^1 \sup_{|\theta|<\pi} |g(r e^{i\theta})| dr < \infty.
\]
Define 
\begin{equation} \label{gtof}
f(z) := \frac{2}{\pi} \int_\D \log \left( \frac{1}{|w|} \right) \frac{zg(w)}{(1-z\overline{w})^2} \,dA(w), \qquad z \in \D.
\end{equation}
This integral is absolutely convergent, and $f \in \B(\D)$, as shown in \cite[Proposition 2.2]{BS22}.

Since the  polynomials are dense in $\B(\D)$, a bounded operator $T$ on a Banach space has a bounded $\B(\D)$-calculus (or $\B$-calculus for short) if and only if there is a constant $C$ such that 
\begin{equation} \label{Best}
\|p(T)\| \le C\|p\|_\B
\end{equation}
for all polynomials $p$.    In particular, $T$ must be power-bounded.    We will present an alternative characterisation of operators $T$ with a bounded $\B$-calculus in terms of a condition on the resolvent of $T$, and this will provide an alternative definition of the calculus.

Assume that $T$ has spectral radius at most $1$.  In order to define $f(T)$ for $f \in \B(\D)$, one might try to replace $z$ by $T$ in \eqref{rep}.   Then we would have an equation involving an operator-valued integral, which might not be absolutely convergent.   In order to make convergence more likely, the operator-valued integral can be replaced by scalar-valued functions, and so we would hope that $f(T)$ would satisfy
\begin{equation} \label{dBfc}
\langle f(T)x, x^* \rangle = f(0) \langle x,x^* \rangle 
 + \frac{2}{\pi} \int_\D \log \left( \frac{1}{|w|} \right) f'(w) \langle T(I-\overline{w}T)^{-2}x, x^* \rangle \,dA(w)
 \end{equation}
for all $x \in X$ and $x^* \in X^*$, with the integrals being absolutely convergent.   We will say that $T$ {\it satisfies the {\rm(dGSF)} condition} if, for all $x\in X$ and $x^* \in X^*$,
\begin{equation} \label{edef}
\sup_{0<r<1} (1-r) \int_{-\pi}^\pi \big| \langle (T(I-re^{i\theta}T)^{-2}x, x^* \rangle \big| \, d\theta < \infty.
\end{equation}
When this holds, the integrals in \eqref{dBfc} are absolutely convergent, and the Closed Graph Theorem shows that the supremum in \eqref{edef} is bounded above by $\gamma_T \|x\|\,\|x^*\|$ for some constant $\gamma_T$.   Then the finiteness of  $\gamma_T$ and $\|f\|_{\B_0}$ imply that there is a bounded linear operator $f(T)$ from $X$ to $X^{**}$ satisfying \eqref{dBfc}.   For $f_k(z) = z^k, \, k\in\Z_+$, it is not difficult to verify that $f_k(T) = T^k$.   Then by linearity, density of the polynomials in $\B(\D)$ and continuity of the map $f \mapsto f(T)$ in the $\B$-norm, it follows that $f(T)$ maps $X$ into $X$ for all $f \in \B(\D)$ and the map $f \mapsto f(T)$ is a (unique) $\B$-calculus for $T$.

The process described above is adapted from a similar process for generators of bounded $C_0$-semigroups in \cite{BGT}, for which the argument was more complicated.  It was shown in \cite{BGT2} that the (GSF) condition, which is the continuous analogue of the (dGSF) condition, is not only sufficient for the existence of a $\B$-calculus, but it is also necessary.  This is also true in the discrete case.   This construction of the $\B$-calculus for power-bounded operators first appeared in the literature in \cite[Theorem 2.10, Propositions 3.1 and 3.3]{Arn21}.  Without being aware of that paper, we wrote a variant of the argument in \cite[Theorems 3.2 and 3.5, Corollary 3.3]{BS22}. 

\begin{thm} \label{BSthm0}
Let $T$ be a bounded linear operator on a Banach space $X$, with spectral radius at most $1$.   Then  $T$ has a (unique) bounded $\B$-calculus if and only if $T$ is power-bounded and satisfies the {\rm (dGSF)} condition.  In that case, the following hold:
\begin{enumerate}[\rm(a)]
\item  {\rm \eqref{dBfc}} is satisfied for all $f \in \B(\D)$, $x \in X$ and $x^* \in X^*$.
\item If $f \in W^+(\D)$, then  $f(T) = \sum_{k=0}^\infty a_k T^k$, with convergence in the operator norm.
\item If $f \in \B(\D)$ with Taylor series  $f(z) = \sum_{k=0}^\infty a_kz^k$, then
\[
f(T) = \lim_{r\to1-} \sum_{k=0}^\infty a_k r^k T^k,
\]
with convergence in the operator norm.
\end{enumerate}
\end{thm}

There is also a spectral mapping theorem for the $\B$-calculus;  see \cite[Theorem 3.8]{BS22}.
 
\begin{exas} \label{exPV}
(a)   Let $T$ be a power-bounded operator on a Hilbert space.   Using Fourier analysis, Peller \cite{Pel82} proved that $T$ satisfies the condition \eqref{Best}, so $T$  has a bounded $\B$-calculus.   Alternatively it is easy to show that $T$ satisfies the (dGSF) condition by using Parseval's formula for functions with values in a Hilbert space.   

(b) A bounded linear operator $T$ on a Banach space is a \emph{Ritt operator} if its spectral radius is at most $1$ and it satisfies the \emph{Ritt resolvent condition}, namely  
\begin{equation}\label{eq:Ritt_res}
\|(zI-T)^{-1}\|\le\frac{C}{|z-1|},\qquad |z|>1,
\end{equation}
for some constant $C$.
Vitse proved in \cite{Vit05} that $T$ has a bounded $\B$-calculus, by using harmonic analysis to establish the estimate \eqref{Best}.    
The approach used above provides a short and explicit way to establish the $\B$-calculus.   Using the estimate \eqref{eq:Ritt_res} it is very easy to show that the (dGSF) condition is satisfied.  
 Moreover, $f(T)$ can be written as an operator-valued integral:
\[
f(T) = f(0)I + \frac{2}{\pi} \int_\D \log \left( \frac{1}{|w|} \right) f'(w) T(I-\overline{w}T)^{-2} \, dA(w).
\]

For any Ritt operator, $\sigma_u(T) \subseteq \{1\}$, so only the special case of Katznelson-Tzafriri theory where $f(z) = 1-z$ is of interest for Ritt operators.  We will discuss that case in more detail in Section \ref{sect3}, and in particular the theory for Ritt operators in Section \ref{sect3.3}.
\end{exas}

The following theorem of Katznelson-Tzafriri type, for operators with a bounded $\B$-calculus, is proved in \cite[Theorem 1.5]{BS22}.   The proof uses the approach to Theorem \ref{KTthm} and Theorem \ref{KvNthm} via the limit isometry, and some manipulations set out in \cite[Section 2]{CG08}.   

\begin{thm} \label{BSthm}
Let $T$ be a power-bounded operator on a Banach space, and assume that $T$ has a bounded $\B$-calculus.   Let $f \in \B(\D)$, and assume that $f$ vanishes on $\sigma_u(T)$.   Then $\lim_{n\to\infty} \|T^nf(T)\|=0$.
\end{thm}

\section{The special  case when $f(T) = I-T$} \label{sect3}

In this section, we consider the case when $f(z) = 1-z$.   This was the case which inspired the authors of \cite{KT86} to obtain Theorem 2.1, and it is the case which is most often used.   

\subsection{The basic result}
Since  $\{1\}$ is a set of spectral synthesis for $W(\T)$, Theorem \ref{KTthm} gives the following.

\begin{cor} \label{KTcor}
Let $T$ be a power-bounded operator on a Banach space.   Then $\lim_{n\to\infty} \|T^n(I-T)\| = 0$ if (and only if) $\sigma_u(T)\subseteq \{1\}$.
\end{cor}

Corollary \ref{KTcor} has many applications.   An elementary application is as follows.

Assume that $T$ is power-bounded on a Banach space $X$, and that $T$ is mean-ergodic (this is automatic if $X$ is reflexive).   Then $X$ decomposes as
\[
X = \overline{\operatorname{Ran}(I-T)} \oplus \Ker(I-T),
\]
and the Ces\`aro means of the sequence $(T^n)_{n\in\Z_+}$ converge to $P$ in the strong operator topology, where $P$ is the projection of $X$ onto the second summand.   

Now assume in addition that $\sigma_u(T) = \{1\}$.   Then the restriction $T_1$ of $T$ to the first summand is power-bounded, with $\sigma_u(T_1) \subseteq \{1\}$.  By Corollary \ref{KTcor} applied to $T_1$, and then by approximation, $\lim_{n\to\infty} \|T^nx\| \to 0$ for all $x$ in the first summand.  Since $T$ acts as the identity on the second summand, it follows that $T^n$ converges in the strong operator topology to the projection $P$ as $n\to\infty$.

There are many different proofs of Corollary \ref{KTcor}, including the following:
\begin{enumerate}[(i)]
\item  Esterle \cite{Est83} proved the result in the special case when $\sigma(T) = \{1\}$.   
\item  Katznelson and Tzafriri  gave a direct proof of Corollary \ref{KTcor} (\cite[Theorem 1]{KT86}), using Tauberian theorems and Fourier analysis.   They also proved Theorem \ref{KTthm} (\cite[Theorem 5]{KT86}) and they noted that Corollary \ref{KTcor} follows.   
\item V\~u \cite{Vu92a} gave a very short proof based on the existence of the limit isometry $S$ whose spectrum is $\{1\}$.   Then Gelfand's theorem implies that $S=I$, and the result follows.  Allan and Ransford \cite[Theorem 1.2]{AR89} had previously given a proof of a version of Corollary \ref{KTcor} for Banach algebras, using some similar ideas to V\~u's proof, but invoking a theorem of Arens.
  \item  Allan, O'Farrell and Ransford \cite{AOR87}  gave a proof based on Tauberian theorems and complex analysis.  Another proof along similar lines was given in \cite[Theorem 3.9]{AP92}.  See also \cite[Corollary 4.7.15]{ABHN}.
  \item  Nevanlinna \cite[Theorem 4.2.2]{Nev93} gave a proof along lines similar to the direct proof in \cite{KT86}.
  \end{enumerate}

The conclusion of Corollary \ref{KTcor} implies that $\sigma_u(T) \subseteq \{1\}$ and $\|T^n\| = o(n)$ as $n\to\infty$.  The question was raised whether the converse is true.   The answer to the question is negative.   This was first observed in a remark following  \cite[Theorem 4.2]{AR89} which established that, for every $\alpha>0$, there is an operator $T$ such that $\sigma_u(T) = \{1\}$ and $\|T^n\| = O(n^\alpha)$ but $\|T^n(I-T)\| \ne O(n^\beta)$ as $n\to \infty$ for any $\beta < \alpha$.   

 Subsequently,  it was shown in \cite[Example 4.1]{TomZem04} that counterexamples exist where $T$ is  Ces\`aro bounded and $X$ is a Hilbert space, and in \cite[Example 4.6]{TomZem04} that counterexamples exist where $\sigma(T) = \{1\}$ and $X$ is a Banach space.    L\'eka \cite{Lek10} subsequently showed that the following operator $T$ is a counterexample satisfying all the additional conditions.

\begin{exa}
Let $V$ be the Volterra operator on $L^2(0,1)$ and let $T$ be the operator 
\[
\begin{pmatrix}  I-V  & -V  \\ 0 & I-V \end{pmatrix}
\]
on the Hilbert space $L^2(0,1) \oplus L^2(0,1)$.   Then $T$ is  Ces\`aro bounded, $\sigma(T) = \{1\}$ and $\lim_{n\to\infty} \|T^n(I-T)\| > 0$.
\end{exa}

On the other hand, there exist bounded operators $T$ which are not power-bounded but satisfy $\lim_{n\to\infty} \|T^n(I-T)\| = 0$ and $\sigma(T) = \{1\}$.  The operator $T: = I-V$ on $L^p(0,1)$ where $p \in [1,\infty], \, p\ne2$, has these properties, but
\[
\| T^n\| \asymp n^{|\frac{1}{4} - \frac{1}{2p}|}, \qquad n\to\infty.
\]
This was shown in \cite{MRSAZ05}, following earlier work by various authors, and more examples were given in \cite{Lek13}.  

A variant of the Katznelson-Tzafriri theorem was obtained in \cite[Theorem 3.4]{Aba16} for operators which are not necessarily power-bounded but are Ces\`aro bounded of some positive order. Here the algebras $W^+(\D)$ and $W(\T)$ are replaced by certain fractional versions of these algebras, and the powers of $T$ in the conclusion of Theorem \ref{KTthm} are replaced by certain fractional Ces\`aro means of $T$.

\subsection{Rates of decay: lower bound}  
The authors of \cite{AR89} noted that their Theorem~4.2 establishes that the rate of decay in Corollary \ref{KTcor} can be arbitrarily slow, in general.   In other words, given any non-negative sequence $(a_n)_{n\ge0}\in c_0$ it is possible to find a power-bounded operator $T$ on a Banach space $X$, such that $\|T^n(I-T)\|\to0$ as $n\to\infty$ but $\|T^n(I-T)\|\ge a_n$ for infinitely many $n\ge0$.

On the other hand, decay rates faster than $n^{-1}$ can happen only in uninteresting cases.   
More precisely, it was shown in \cite[Theorem~3.1]{KMSOT04} that if 
\[
\limsup_{n\to\infty} \,n\|T^n(I-T)\| < e^{-1},
\]
then $X$ splits as $X = X_1 \oplus X_2$, where $X_1$ and $X_2$ are closed $T$-invariant subspaces of $X$, and $Tx=x$ for all $x\in X_1$, and the restriction of $T$ to $X_2$ has spectral radius less than $1$.  Thus $T^n$ converges in the operator norm to a projection and $T$ is power-bounded. By \cite[Theorem~3.4]{KMSOT04} this statement is optimal in the rather strong sense that on any infinite-dimensional Banach space $X$ we may define an equivalent norm in such a way that there exists an operator $T$ on $X$ which fails to be power-bounded but satisfies $\limsup_{n\to\infty}n\|T^n(I-T)\|=e^{-1}$.
An earlier result due to Esterle \cite{Est83} had shown that if $\sigma(T)=\{1\}$ but $T\ne I$, then $$\liminf_{n\to\infty}\,n\|T^n(I-T)\|\ge\frac{1}{96}.$$ The constant $1/96$ was improved to $1/12$ by Berkani in his thesis \cite{Ber83}, and then to $e^{-1}$ in \cite[Theorem~2.2]{KMSOT04}, which is the optimal value (even if one restricts to Hilbert space) by \cite[Theorem~2.3(2)]{KMSOT04}.

\subsection{Ritt operators}  \label{sect3.3}
We return to the Ritt operators which were briefly discussed in Example \ref{exPV}(b).

 It is clear that the spectrum of any Ritt operator $T$ must satisfy $\sigma(T)\subseteq\D\cup\{1\}$. The case when $1\not\in\sigma(T)$ is of no interest to us, since then $T^n$ decays exponentially to zero in the operator norm as $n\to\infty$. When $1\in\sigma(T)$, the elementary resolvent estimate 
$$\|(zI-T)^{-1}\|\ge\frac{1}{\mathrm{dist}(z,\sigma(T))},\qquad z\in\rho(T),$$ shows that the Ritt resolvent condition~\eqref{eq:Ritt_res} gives the sharpest possible upper bound that the resolvent can satisfy (up to the choice of $C$). It was first shown by Komatsu \cite{Kom68} that the Ritt resolvent condition implies that $T$ is power-bounded and  $\|T^n(I-T)\|=O(n^{-1})$ as $n\to\infty$; see also \cite{Bak88} or \cite[Theorem~4.5.4]{Nev93}. Their arguments are based on contour integrals. The converse was proved in \cite[Theorem~2.1]{Nev97}; see also \cite{Lyu99, NagZem99}. The Ritt resolvent condition goes back to \cite{Rit53}, where it is shown that the condition implies $\|T^n\|=o(n)$ as $n\to\infty$. Another early appearance of the Ritt resolvent condition~\eqref{eq:Ritt_res} was in  \cite{Tad86}, which may explain why some authors refer to Ritt operators as \emph{Tadmor-Ritt operators}.

It is shown in \cite[Theorem~3.3]{KMSOT04} by means of a suitable Fourier multiplier operator on $L^1(\R)$ that a bounded linear operator $T$ for which $\|T^n(I-T)\|=O(n^{-1})$ as $n\to\infty$ need not be power-bounded; in the example  $\|T^n\|$ grows like $\log(n)$ as $n\to\infty$. The first example of a Ritt operator $T$ satisfying $\sigma(T)=\{1\}$ was given by Lyubich~\cite{Lyu01}, and L\'eka~\cite{Lek13} has constructed further examples. 

Dungey showed in \cite[Theorem~1.3]{Dun11} that a bounded linear operator $T$ is a Ritt operator if and only if $T=I-(I-S)^\alpha$ for some $\alpha\in(0,1)$ and some power-bounded operator $S$, and one obtains a further equivalent condition if one replaces power-boundedness of $S$ by the condition that $S$ be a \emph{Kreiss operator}, that is, its spectral radius is at most $1$ and 
$$\|(zI-S)^{-1}\|\le\frac{C}{|z|-1},\qquad |z|>1,$$
for some constant $C$. It should be noted that $I-S$ is sectorial whenever $S$ is a Kreiss operator, and in particular the fractional powers $(I-S)^\alpha$ for $\alpha>0$ are well defined. Every power-bounded operator is a Kreiss operator, but the converse is far from true; see for instance~\cite[Theorem~6]{Nev01}.

\subsection{Rates of decay: polynomial rates}
Ritt operators exhibit the fastest possible polynomial rate of decay in Corollary~\ref{KTcor}. Beyond this extremal case
 there has been considerable interest in power-bounded operators $T$ for which $\|T^n(I-T)\|=O(n^{-\beta})$ as $n\to\infty$ for some $\beta\in(0,1)$. Examples with `large' spectra for which $\|T^n(I-T)\|\asymp n^{-\beta}$ as $n\to\infty$ with arbitrary $\beta\in(0,1)$ are easy to construct, for instance by considering suitable normal operators, or more generally multiplication operators on classical function spaces. On the other hand, examples with minimal spectrum, that is, $\sigma(T)=\{1\}$, are much harder to construct. By considering carefully chosen functions of the Volterra operator $V$ on $L^2(0,1)$, L\'eka~\cite{Lek14} showed that for every $\beta\in(1/2,1)$ one may find a contraction $T$ on a Hilbert space such that $\sigma(T)=\{1\}$ and $\|T^n(I-T)\|\asymp n^{-\beta}$ as $n\to\infty$.  This answered in the negative a question raised by Zem\'anek in~\cite{Zem07}, namely whether a power-bounded operator $T$ on a Banach space satisfying $\sigma(T)=\{1\}$ and $\|T^n(I-T)\|=O(n^{-\beta})$ as $n\to\infty$ for some $\beta\in(1/2,1)$ must satisfy  $\|T^n(I-T)\|=O(n^{-1})$ as $n\to\infty$ (and hence be a Ritt operator). As noted by L\'eka in his survey~\cite{Lek17}, no examples are known, even on Banach spaces,  of power-bounded operators $T$ for which $\sigma(T)=\{1\}$ and $\|T^n(I-T)\|\asymp n^{-\beta}$ as $n\to\infty$ for some $\beta\in(0,1/2)$. 

 The case $\beta=1/2$ in the preceding discussion is in some sense special, and at any rate has received particular attention. A prominent example of an operator exhibiting this particular rate of decay is $T:=I-V$, where $V$ is the Volterra operator on $L^2(0,1)$. Indeed, $T$ is power-bounded and satisfies $\sigma(T)=\{1\}$ and $\|T^n(I-T)\|\asymp n^{-1/2}$ as $n\to\infty$; see \cite{MRSAZ05, Tse03}. Foguel and Weiss~\cite[Lemma~2.1]{FogWei73}  and later Nevanlinna~\cite[Theorem~4.5.3]{Nev93} showed that an operator $T$ is power-bounded and satisfies $\|T^n(I-T)\|=O(n^{-1/2})$ as $n\to\infty$ whenever $T=\alpha I+(1-\alpha)S$ for some power-bounded operator $S$ and some $\alpha\in(0,1)$. In fact, \cite[Lemma~2.1]{FogWei73} is stated only for the case of contractions, but it extends straightforwardly to power-bounded operators. 
 More recently Dungey proved in \cite[Theorem~1.2]{Dun08} that the above implication is in fact an equivalence, and he gave several further equivalent conditions. 
 
\subsection{Rates of decay: general results}  \label{sect3.4}
 
In view of the fact that for power-bounded operators $T$ the (extremal) decay rate $\|T^n(I-T)\|=O(n^{-1})$ as $n\to\infty$ in Corollary~\ref{KTcor} is characterised in terms of the resolvent condition~\eqref{eq:Ritt_res}, it is natural to ask whether the decay rate in Corollary~\ref{KTcor} can be related to properties of the resolvent of $T$ more generally. The following is a general result of this type, giving a decay rate for $\|T^n(I-T)\|$ as $n\to\infty$ for a power-bounded operator $T$ satisfying $\sigma_u(T)\subseteq\{1\}$ depending on the growth of the resolvent norms $\|(e^{i\theta}I-T)^{-1}\|$ as $\theta\to0$.

\begin{thm}\label{thm:mlog}
Let $T$ be a power-bounded operator on a Banach space, and assume that $\sigma_u(T)\subseteq\{1\}$. Let $m:(0,\pi]\to(0,\infty)$ be a continuous non-increasing  function such that $ \|(e^{i\theta}I-T)^{-1}\|\le m(|\theta|)$ for $0<|\theta|\le\pi$. Then, for any $c\in(0,1)$,
\begin{equation}\label{eq:mlog}
\|T^n(I-T)\|=O\big(m_{\rm log}^{-1}(cn)\big),\qquad n\to\infty,
\end{equation}
where $m_{\rm log}^{-1}$ is the inverse of the function $m_{\rm log}$ defined by
$$m_{\rm log}(s):=m(s)\log\left(1+\frac{m(s)}{s}\right),\qquad 0< s\le\pi.$$
\end{thm}

Two proofs of Theorem~\ref{thm:mlog} have been given. The first, in \cite[Theorem~2.1]{Sei16}, uses contour integrals while the second, given in \cite[Theorem~2.5]{Sei15b}, uses techniques from Fourier analysis. In the special case when $ \|(e^{i\theta}I-T)^{-1}\|\le C|\theta|^{-\alpha}$ for some constants $C>0$, $\alpha\ge1$ and all $\theta$ such that $0 < |\theta| \le\pi$,~\eqref{eq:mlog} becomes
\begin{equation}\label{eq:poly_log}
\|T^n(I-T)\|=O\bigg(\frac{(\log n)^{1/\alpha}}{n^{1/\alpha}}\bigg),\qquad n\to\infty.
\end{equation}
A less sharp result of similar type had previously been obtained by Nevanlinna in~\cite[Theorem~9]{Nev01}.
In fact, the upper bound in \eqref{eq:poly_log} is sharp in general. Indeed, it was shown in \cite[Theorem~3.6]{Sei16} that for any $\alpha>2$ there exists a complex Banach space $X_\alpha$ and a power-bounded operator $T_\alpha$ on $X_\alpha$ such that $\sigma_u(T_\alpha)=\{1\}$, $\|(e^{i\theta}I-T)^{-1}\|=O(|\theta|^{-\alpha})$ as $\theta\to0$ and
$$\limsup_{n\to\infty}\frac{n^{1/\alpha}}{\log(n)^{1/\alpha}}\|T_\alpha^n(I-T_\alpha)\|>0.$$ 
On the other hand, for $\alpha=1$ the estimate~\eqref{eq:poly_log} cannot be sharp, since in this case it is easy to see that $T$ satisfies the Ritt resolvent condition~\eqref{eq:Ritt_res} and hence $\|T^n(I-T)\|=O(n^{-1})$ as $n\to\infty$. It is not known whether there exist examples with the same properties as the operator $T_\alpha$ above when $\alpha\in(1,2]$.

The preceding remarks motivate the question whether the upper bounds in~\eqref{eq:mlog} and~\eqref{eq:poly_log}  can be improved at least in some cases and, if so, by how much. Given a power-bounded operator $T$ such that $\sigma_u(T)\subseteq\{1\}$, let 
\begin{equation}\label{eq:m_min}
m(s):=\max_{s\le|\theta|\le\pi}\|(e^{i\theta}I-T)^{-1}\|,\qquad s\in(0,\pi].
\end{equation}
This is the smallest continuous non-increasing  function $m:(0,\pi]\to(0,\infty)$ such that $\|(e^{i\theta}I-T)^{-1}\|\le m(|\theta|)$ for $0<|\theta|\le\pi$. Under the assumption that 
\begin{equation}\label{eq:tec_res}
\lim_{\theta\to0}\max\big\{\|\theta(e^{i\theta}I-T)^{-1}\|,\|\theta(e^{-i\theta}I-T)^{-1}\|  \big\}=\infty,
\end{equation}
it was shown in \cite[Corollary~2.6]{Sei16} that 
\begin{equation}\label{eq:KT_lb}
\|T^n(I-T)\|\ge c m^{-1}(Cn)
\end{equation}
for some constants $C$, $c$ and all sufficiently large $n$. As noted in \cite[Remark~2.7]{Sei16}, it is possible to weaken the technical assumption~\eqref{eq:tec_res} but some condition of this kind must be retained, since otherwise the identity operator gives a counterexample. A lower bound closely related to~\eqref{eq:KT_lb} was obtained in~\cite[Proposition~2.1]{NS20}. It follows from~\eqref{eq:KT_lb} that if the upper bound $m$ in Theorem~\ref{thm:mlog} is chosen optimally then the resulting bound~\eqref{eq:mlog} is never suboptimal by more than the logarithmic correction arising from the definition of $m_{\rm log}$. In the case of polynomial resolvent growth, if we assume that $\|(e^{i\theta}I-T)^{-1}\|\asymp|\theta|^{-\alpha}$ as $\theta\to0$ for some $\alpha>1$, then~\eqref{eq:KT_lb} becomes $\|T^n(I-T)\|\ge cn^{-1/\alpha}$ for some constant $c$ and all sufficiently large $n$. In particular, the upper bound in~\eqref{eq:poly_log} is suboptimal by at most a logarithmic factor in this case.

When the underlying space is a Hilbert space, it is often possible to remove the logarithmic factor in~\eqref{eq:poly_log}, and hence, according to~\eqref{eq:KT_lb}, obtain the best possible rate of decay in Corollary~\ref{KTcor}.   This was proved in~\cite[Theorem~3.10]{Sei16} for the case when $m(s) = Cs^{-\alpha}$ for some $\alpha>1$, by Fourier-analytic techniques.    We state this case in Theorem \ref{thm:poly_decay}, and we give a different proof based on a simple trick used in~\cite[Theorem~2.4]{BT10} in the continuous-parameter setting. This argument was sketched in \cite[Remark~3.11]{Sei16} (and spelt out in \cite[pp.~99--100]{Sei14}).  We will state the general result in Theorem \ref{thm:m_inv}.

\begin{thm}\label{thm:poly_decay}
Let $T$ be a power-bounded operator on a Hilbert space $X$, such that $\sigma_u(T)\subseteq\{1\}$. Assume that $\|(e^{i\theta}I-T)^{-1}\|=O(|\theta|^{-\alpha})$ as $\theta\to0$, for some $\alpha\ge1$. Then
\begin{equation}\label{eq:poly_decay}
\|T^n(I-T)\|=O\big(n^{-1/\alpha}\big),\qquad n\to\infty.
\end{equation}
\end{thm}

\begin{proof}
Since $T$ is power-bounded the operator $I-T$ is sectorial, so we may consider its fractional powers. A straightforward application of the moment inequality shows that~\eqref{eq:poly_decay} is equivalent to $\|T^n(I-T)^\alpha\|=O(n^{-1})$ as $n\to\infty$. Consider the Hilbert  space $Z=X\times X$ and define $Q\in L(Z)$ by $Q(x,y)=(Tx+T(I-T)^\alpha y,Ty)$ for $(x,y)\in Z$. Note that the powers of $Q$ are represented by the operator matrices
$$Q^n=\begin{pmatrix}
T^n &nT^n(I-T)^\alpha\\ 0&T^n
\end{pmatrix},\qquad n\ge0.
$$
In particular, $\|T^n(I-T)^\alpha\|=O(n^{-1})$ as $n\to\infty$ if and only if $Q$ is power-bounded. We now prove that $Q$ is power-bounded.

 It was shown in \cite[Corollary~2.5]{CG08} (see also \cite[Proposition~6.4.4]{Sei14}) that a bounded linear operator $S$ on a Hilbert space $Y$ is power-bounded if and only if $\sigma(S) \subseteq \overline{\D}$  and
\begin{equation}\label{eq:GSF}
\sup_{r>1}\,\,(r-1)\int_{-\pi}^\pi\big(\|(re^{i\theta}I-S)^{-1}y\|^2+\|(re^{i\theta}I-S^*)^{-1}y\|^2\big)\,\dd \theta<\infty
\end{equation}
for all $y\in Y$, where $S^*$ denotes the (Hilbert space) adjoint of $S$. Returning to the operators $T$ on $X$ and $Q$ on $Z$, we have $\sigma(Q)=\sigma(T) \subseteq \overline{\D}$.  Moreover a straightforward calculation gives
$$(\l I-Q)^{-1}=\begin{pmatrix}
(\l I-T)^{-1} &(\l I-T)^{-2} (I-T)^\alpha T\\ 0&(\l I-T)^{-1} 
\end{pmatrix},\qquad \lambda\in\rho(Q).$$
Hence if $z=(x,y)\in Z$ then 
$$\begin{aligned}
\|(\l I-Q)^{-1} z\|&\le\|(\l I-T)^{-1} x\|+\|(\l I-T)^{-1} y\|\\&\qquad +\|(\l I-T)^{-1} (I-T)^\alpha\| \|T\|\|(\l I-T)^{-1} y\|
\end{aligned}$$
for $|\lambda|>1$. 
By \cite[Lemma~3.9]{Sei16}, our assumption that $\|(e^{i\theta}I-T)^{-1} \|=O(|\theta|^{-\alpha})$ as $\theta\to0$ implies that $\sup_{|\lambda|>1}\|(\l I-T)^{-1} (I-T)^\alpha\|<\infty$. Thus there exists a constant $C$ such that 
$$\|(\l I-Q)^{-1} z\|^2\le C\big(\|(\l I-T)^{-1} x\|^2+\|(\l I-T)^{-1} y\|^2\big),\qquad|\lambda|>1,$$
and the same bound holds when $Q$ and $T$ are replaced by their adjoints.
Since $T$ is power-bounded, it follows from~\eqref{eq:GSF} with  $S=T$
that \eqref{eq:GSF} holds also for $Q$. Hence $Q$ is power-bounded, as required.
\end{proof}

Theorem~\ref{thm:poly_decay} shows that if $X$ is a Hilbert space and  $m(s)=Cs^{-\alpha}$, $0<s\le\pi$, for some constants $C,\alpha\ge1$, then
 ~\eqref{eq:mlog} in Theorem~\ref{thm:mlog} can be strengthened to 
$$\|T^n(I-T)\|=O\big(m^{-1}(n)\big),\qquad n\to\infty,$$
By~\eqref{eq:KT_lb} this is in general the best estimate (up to the choice of constants) that can hold. This optimal decay result for power-bounded operators on Hilbert spaces was extended in~\cite{NS20} from the polynomial case to the much wider class of functions $m$ having \emph{reciprocally positive increase}. A continuous non-increasing function $m:(0,\pi]\to(0,\infty)$ is said to have reciprocally positive increase if there exist constants $c\in(0,1]$, $\alpha>0$ and $s_0\in(0,\pi]$ such that
$$\frac{m(t^{-1}s)}{m(s)}\ge ct^\alpha,\qquad t\ge1,\ 0<s\le s_0.$$
The condition is equivalent to the function $s\mapsto m(s^{-1})$ having \emph{positive increase}, which is a classical notion, and it may be understood as requiring the growth of $m(s)$ to be sufficiently regular as $s\to0+$. The class of functions of reciprocally positive increase is vast.   In particular, any continuous non-increasing function $m:(0,\pi]\to(0,\infty)$ which is of the form $m(s)=Cs^{-\alpha}|\log s|^\beta$ for all sufficiently small values of $s$, where $\alpha>0$ and $\beta\in\R$ are constants, has reciprocally positive increase. The following result was proved in~\cite[Theorem~2.3]{NS20}, using techniques from Fourier analysis.
 
\begin{thm}\label{thm:m_inv}
Let $T$ be a power-bounded operator on a Hilbert space, such that $\sigma_u(T)\subseteq\{1\}$. Assume that $\|(e^{i\theta}I-T)^{-1}\|\le m(|\theta|)$ for $0<|\theta|\le\pi$, where $m:(0,\pi]\to(0,\infty)$ is a function of reciprocally positive increase. Then
\begin{equation}\label{eq:m_inv}
\|T^n(I-T)\|=O\big(m^{-1}(n)\big),\qquad n\to\infty,
\end{equation}
\end{thm}

Here $m^{-1}$ is the  right-inverse of $m$ defined by
$$m^{-1}(t)=\inf\{s\in(0,\pi]:m(s)\le t\},\qquad t\ge m(\pi).$$ Note that, in contrast to~\eqref{eq:mlog}, there is no constant $c$ appearing in~\eqref{eq:m_inv}. This is because functions of reciprocally positive increase are characterised by the property that $m^{-1}(cn)\asymp m^{-1}(n)$ as $n\to\infty$ for all $c>0$; see~\cite[Proposition~2.2]{RSS}. It is clear from the preceding discussion that~\eqref{eq:m_inv} gives the best possible rate one might hope for in Corollary~\ref{KTcor}. Theorem~\ref{thm:m_inv} is optimal in a further sense, namely that the class of functions having reciprocally positive increase is the largest class for which the conclusion of the theorem can hold. Indeed, it was shown in~\cite[Proposition~4.1]{NS20} that if $T$ is a power-bounded normal operator on a Hilbert space such that $\sigma_u(T)\subseteq\{1\}$ and $\|T^n(I-T)\|=O(m^{-1}(cn))$ as $n\to\infty$ for some $c>0$ and $m$ defined as in~\eqref{eq:m_min}, then $m$ must have reciprocally positive increase. See~\cite[Example~3.1]{NS20} for an example of a (non-normal) operator $T$ on a Hilbert space for which the conditions of Theorem~\ref{thm:m_inv} are satisfied for a function $m$ that has reciprocally positive increase but is not of polynomial form. In this example,~\eqref{eq:m_inv} improves on~\eqref{eq:mlog} by a factor of $\log n$.

\section{Applications}  \label{sect4}

The Katznelson-Tzafriri theorem and its variants and refinements discussed in Sections~\ref{sect2} and~\ref{sect3} have been applied in various areas of mathematics. In this section we briefly discuss some of these applications. 
We make no attempt at being exhaustive in our coverage, and our selection is  guided by our own interests.

\subsection{The countable spectrum theorem}\label{sec:ABLV}

 The ``countable spectrum theorem'' is a well-known result in the asymptotic theory of operator semigroups. The version of the theorem  most relevant to us may be stated as follows.
 
 \begin{thm}\label{thm:ABLV}
Let $T$ be a power-bounded operator on a Banach space $X$, and assume that $\sigma_u(T)$ is countable and contains no eigenvalues of the dual $T^*$ of $T$. Then $\lim_{n\to\infty}\|T^nx\|=0$ for all $x\in X$.
 \end{thm}

There exists an analogue of Theorem~\ref{thm:ABLV} for bounded $C_0$-semigroups of operators, which was proved independently by Arendt and Batty~\cite{AB88} and by Lyubich and V\~u~\cite{LyuVu88}, and the discrete version of the countable spectrum theorem stated above was proved in \cite[Theorem~5.1]{AB88}.   Both the discrete and the continuous version of the countable spectrum theorem are also known as the ``ABLV theorem''. 

The methods of proof given in \cite{AB88} and~\cite{LyuVu88} are very different. The former uses transfinite induction and a general Tauberian theorem, whereas the latter uses a shorter, purely functional analytic argument. Soon after the appearance of these two papers, Esterle, Strouse and Zouakia~\cite[Section~3]{ESZ1} gave a third proof of Theorem~\ref{thm:ABLV} based on the Katznelson-Tzafriri theorem. The authors considered the ideal
$$I:=\big\{f\in W^+(\D): f\mbox{ vanishes on } \sigma_u(T)\big\}$$
of $W^+(\D)$, and they were able to prove that under the assumptions of Theorem~\ref{thm:ABLV} the set
$$Y:=\big\{f(T)x:f\in I,\ x\in X\big\}$$
spans a dense subspace of $X$. Their argument relied on the theory of almost periodic sequences and a powerful result of Mittag-Leffler type. Since closed countable subsets of $\T$ are of spectral synthesis for $W(\T)$, any element of $I$ is of spectral synthesis in $W(\T)$ with respect to $\sigma_u(T)$. It follows from the Katznelson-Tzafriri theorem (Theorem~\ref{KTthm}) that $\|T^ny\|\to0$ for all $y\in Y$. Now Theorem~\ref{thm:ABLV} follows by density of the linear span of $Y$ in $X$ and power-boundedness of $T$.

\subsection{The method of alternating projections}

Let $Y_1,\dots,Y_N$ be a collection of $N\ge2$ closed subspaces of a Hilbert space $X$. Let $P_k$ be the orthogonal projection onto $Y_k$ for $k=1,\dots, N$, and let $P$ be the orthogonal projection onto $Y=Y_1\cap\dotsc\cap Y_N$. Given $x\in X$ we may inductively define a sequence $(x_n)_{n\ge0}$ in $X$ by setting $x_0=x$ and $x_n=P_nx_{n-1}$ for $n\ge1$, where the subscript $n$ of $P_n$ cycles through $\{1,\dots,N\}$. This process is known as the ``method of alternating projections''. It was shown by Halperin~\cite{Hal62} that $\lim_{n\to\infty}\|x_n-Px\|=0$. Halperin's theorem has applications in many areas of mathematics, including image reconstruction, computed tomography, and domain decomposition methods in the solution of elliptic partial differential equations on composite domains; see also \cite[Section~7]{Deu92}.

Halperin's theorem may be restated as saying that $\lim_{n\to\infty}\|T^nx-Px\|=0$ for all $x\in X$, where $T\in L(X)$ is the contraction defined by $T:=P_N\dots P_1$. It is easy to see that $\Ker(I-T)=Y$, and it follows from the mean ergodic theorem that 
$$X=\overline{\operatorname{Ran}(I-T)}\oplus Y.$$
Now Halperin's theorem is a simple consequence of the Katznelson-Tzafriri theorem in the form of Corollary~\ref{KTcor}. Motivated by its many applications, there has been considerable interest in the \emph{rate} of convergence in Halperin's theorem. Badea, Grivaux and M\"uller~\cite{BGM11} proved that the following dichotomy holds: Either the rate of convergence is exponentially fast in the sense that
\begin{equation}\label{eq:MAP_exp}
\|T^n-P\|=O(r^n),\qquad n\to\infty,
\end{equation}
 for some $r\in(0,1)$, or the convergence is arbitrarily slow in the sense that for every $\ep>0$ and every non-negative sequence $(a_n)_{n\ge0}$ in $c_0$ there exists $x\in X$ such that $\|x\|<\sup_{n\ge0}|a_n|+\ep$ and $\|T^nx-Px\|\ge a_n$ for all $n\ge0$. Moreover, we have uniform exponential convergence if and only if any of the following equivalent conditions are satisfied:
\begin{enumerate}[(i)]
\item $\operatorname{Ran}(I-T)$ is closed in $X$;
\item $Y_1^\perp+\dots +Y_N^\perp$ is closed in $X$;
\item The Friedrichs number $c(Y_1,\dots,Y_N)$ is strictly smaller than 1.
\end{enumerate}
Here the {\em Friedrichs number} $c(Y_1,\dots,Y_N)$ is a real number between 0 and 1 which describes the geometric relationship between the subspaces $Y_1,\dots, Y_N$. For instance, when $X=\R^2$ and $Y_1$, $Y_2$ are lines through the origin, then $c(Y_1,Y_2)$ coincides with the cosine of the (smaller) angle made by the lines. When $c(Y_1,\dots,Y_N)<1$ it is possible to obtain a bound for the number $r$ in \eqref{eq:MAP_exp} in terms of the Friedrichs number $c(Y_1,\dots,Y_N)$, and also in terms of other geometric quantities relating to the subspaces $Y_1,\dots, Y_N$; see for instance~\cite[Section~4B]{BGM11}.

Some of the results in~\cite{BGM11} have subsequently been refined. It was noted in~\cite[Proposition~4.1]{BadSei16} that the numerical range $W(T):=\{\langle Tx,x\rangle:x\in X,\|x\|=1\}$ of $T$ is contained in a Stolz domain, that is to say the convex hull of $r\overline{\D}\cup\{1\}$ for some $r\in(0,1)$.  Now the elementary estimate
$$\|(\lambda I-T)^{-1}\|\le \frac{1}{\mathrm{dist}(\lambda,W(T))},\qquad \lambda\in\C\setminus W(T),$$
shows that $T$ is a Ritt operator. (The idea of using the numerical range to obtain resolvent estimates for operators on Hilbert spaces was used in a similar way in \cite{CohLin16}.) In fact, as a consequence of \cite[Theorem~2.2]{BadSei16}, $T$ is even an unconditional Ritt operator, that is, there is a constant $C$ such that, for any $n\in\N$ and  $a_0,\dots,a_n\in\C$,
$$\bigg\|\sum_{k=0}^na_k T^k(I-T)\bigg\|\le C\max_{0\le k\le n}|a_k|.$$
Using the fact that $\|T^n(I-T)\|=O(n^{-1})$ as $n\to\infty$, it was shown in~\cite[Theorem~4.3]{BadSei16} (by means of a Mittag-Leffler argument similar to the one mentioned in Section~\ref{sec:ABLV}) that even if the subspaces $Y_1,\dots, Y_N$ satisfy $c(Y_1,\dots, Y_N)=1$, so that the rate of convergence in Halperin's theorem is arbitrarily slow, then there exists a dense subspace $Z$ of $X$ such that, for every $x\in Z$, the decay is rapid in the sense that $\|T^nx-Px\|=O(n^{-\beta})$ as $n\to\infty$ for all $\beta >0$. We refer the reader to~\cite{BadSei17} for related operator-theoretic results, and to \cite{BorKop20} for an approach  to rates of convergence in the method of alternating projections based on greedy approximation. The latter paper in particular includes a resolution of the well-known Deutsch-Hundal conjecture.

\subsection{Non-autonomous evolution equations with periodic damping} 

The Katznelson-Tzafriri theorem and its quantified variants  have, perhaps surprisingly,  also found applications in the study of continuous-time dynamical systems. In this subsection and the next we briefly discuss two examples of this kind. Consider the non-autonomous evolution equation
\begin{equation}\label{eq:ACP}
\left\{\begin{aligned}
\dot{z}(t)&=A(t)z(t),\qquad t\ge0,\\
z(0)&=x\in X,
\end{aligned}\right.
\end{equation}
where $X$ is a Hilbert space and $(A(t))_{t\ge0}$ is a family of closed operators on $X$ of the form
$$A(t)=A_0+B(t)B(t)^*,\qquad t\ge0.$$
Here $A_0$ is assumed to be the generator of a unitary group $(T_0(t))_{t\in\R}$ on $X$, and $B\in L^2_{\rm loc}(\R_+;L(V,X))$ for another Hilbert space $V$. If $(U(t,s))_{t\ge s\ge0}$ is the evolution family associated with $(A(t))_{t\ge0}$, then the solution of~\eqref{eq:ACP}, in a natural ``mild'' sense, is given by $z(t)=U(t,0)x$ for $t\ge0$; see~\cite[Section~3]{PauSei19}. 

If the family $(A(t))_{t\ge0}$ is periodic, with period $\tau$ say, then one may study the long-time asymptotic behaviour of solutions to~\eqref{eq:ACP} by considering the orbits of the monodromy operator $T:=U(\tau,0)$. Under the additional assumption that  $T_0(\tau)=I$, a requirement referred to in~\cite{PauSei19} as a ``resonance condition'', it was proved in~\cite[Proposition~3.4]{PauSei19} that $T$ is a Ritt operator. Using this observation (and the same Mittag-Leffler argument as in the previous subsections) it was shown in~\cite[Theorem~3.7]{PauSei19} that, under the resonance condition, all solutions of \eqref{eq:ACP} converge as $t\to\infty$ to a periodic solution  of \eqref{eq:ACP}, and that for a dense subspace of initial values $x\in X$ the rate of convergence occurs ``superpolynomially fast'' in the sense that the difference decays faster than $t^{-\beta}$ as $t\to\infty$ for all $\beta>0$. The same theorem also shows that the convergence is in fact exponentially fast, uniformly in the initial data, provided a certain observability condition is satisfied. These abstract results are then applied in~\cite[Sections~4 and~5]{PauSei19} to study the classical transport and wave equations in one spatial dimension, subject to appropriate periodic damping.

\subsection{Infinite systems or coupled ordinary differential equations}

Another example where a good understanding of discrete semigroups can shed light on the behaviour of continuous-time systems is in the study of certain (infinite) systems of coupled ordinary differential equations. Let 
 $m\in\N$ and let $A_0$, $A_1$ be $m\times m$ matrices satisfying the relation
\begin{equation}\label{eq:char}
A_1(\lambda I-A_0)^{-1}A_1=\phi(\lambda)A_1,\qquad \lambda\in\rho(A_0),
\end{equation}
for some rational function $\phi$. Note that in this case $|\phi(\lambda)|\to0$ as $|\lambda|\to\infty$, and the poles of $\phi$ form a (possibly strict) subset of $\sigma(A_0)$. We consider the system of coupled ordinary differential equations given by
\begin{equation}\label{eq:sys}
\dot{x}_k(t)=A_0x_k(t)+A_1x_{k-1}(t),\qquad t\ge0,\ k\in\Z.
\end{equation}
We may think of this system as describing the behaviour of infinitely many agents, indexed by $k\in\Z$, whose states $x_k(t)\in \C^m$ evolve in time as a function of their own current state as well as that of the preceding agent. Such models arise for instance in the study of (infinitely) long chains of vehicles, known sometimes in the control theory literature as ``platoon models''. 

In general, it is difficult to say much at all about the asymptotic behaviour of this simple-looking coupled system. However, condition~\eqref{eq:char} makes it possible to say a considerable amount (and the assumption is satisfied in the main cases of interest to control theorists). When~\eqref{eq:char} holds we call $\phi$ the {\it characteristic function} of our system. The class of coupled systems possessing a characteristic function was first studied in \cite{PauSei17}, with further developments in \cite{PauSei18, PauSei20}. We let $1\le p\le\infty$ and rewrite the system~\eqref{eq:sys} in the form of an abstract Cauchy problem,
\begin{equation}\label{eq:ACP2}
\left\{\begin{aligned}
\dot{z}(t)&=Az(t),\qquad t\ge0,\\
z(0)&=x\in \ell^p(\Z;\C^{m}),
\end{aligned}\right.
\end{equation}
where $z(t)=(x_k(t))_{k\in\Z}$ and $Ay=(A_0y_k+A_1y_{k-1})_{k\in\Z}$ for $y=(y_k)_{k\in\Z}\in\ell^p(\Z;\C^m)$. It turns out that the asymptotic behaviour of the solutions of~\eqref{eq:sys} can be described precisely in terms of the properties of the characteristic function. As an important preliminary result, it was shown in \cite[Theorem~2.3]{PauSei17} that the spectrum of $A$ satisfies $\sigma(A)\setminus\sigma(A_0)=\Omega_\phi$, where
$$\Omega_\phi=\big\{\lambda\in\rho(A_0):|\phi(\lambda)|=1\big\}.$$
This result is the starting point for an in-depth analysis of the quantified asymptotic behaviour of the solution $z(t)$ as $t\to\infty$ based on the (pseudo)spectral properties of $A$. 

We state only two representative results. In order to do so, recall that a smooth  function $f:(0,\infty)\to\R$ is said to be {\it completely monotonic} if $(-1)^nf^{(n)}(s)\ge0$ for all $s\in(0,\infty)$ and all $n\in\Z_+$. We shall say that a rational function $\psi$ is {\it totally monotonic} if there exist $\ep>0$ and a sequence $(a_n)_{n\ge0}\in\ell^1$ of non-negative terms such that 
$$\psi\left(\frac{\lambda-1}{\ep}\right)=\sum_{n=0}^\infty\frac{a_n}{\lambda^n},\qquad |\lambda|\ge1.$$
It can be shown that the restriction to $(0,\infty)$ of any totally monotonic function is completely monotonic. We may now state our two asymptotic results as follows. 

\begin{thm}
Suppose that $\sigma(A_0)\subseteq\C_-:=\{\lambda\in\C:\Re\lambda<0\}$ and $\{0\}\subseteq\Omega_\phi\subseteq\C_-\cup\{0\}$, and assume that the restriction to $(0,\infty)$ of the characteristic function $\phi$ is completely monotonic. Then
$$\|\dot{z}(t)\|=O\left(\frac{(\log t)^{|\frac{1}{2}-\frac{1}{p}|}}{t^{1/2}}\right),\qquad t\to\infty,$$
for all initial data $x\in\ell^p(\Z;\C^m)$. 

Moreover, if $\phi$ is totally monotonic and of the form $\phi(\lambda)=p_0(0)/p_0(\lambda)$ for $\lambda\in\rho(A_0)$, where $p_0$ is the characteristic polynomial of $A_0$, then $\|\dot{z}(t)\|=O(t^{-1/2})$ as $t\to\infty$, and this rate is optimal. 
\end{thm}

Both of these statements follow from \cite[Theorem~2.15]{PauSei20}. The first statement is a consequence of results in the quantified asymptotic theory of $C_0$-semigroups which may be viewed as analogues of the results discussed in Section~\ref{sect3.4}, together with an application of the Riesz-Thorin interpolation theorem. The second statement is even more closely related to the subject matter of this survey. Indeed, in order to establish the sharp rate $t^{-1/2}$ as $t\to\infty$, it was observed in \cite{PauSei20} that the assumptions ensure power-boundedness of the operator $S_\ep:=\ep A+I$ for some $\ep>0$. It follows from a result quoted at the end of Section~\ref{sect3.3} that the operator 
$$T:=\frac\ep2 A+I=\frac12 S_\ep+\frac12 I$$ is power-bounded and satisfies $\|T^n(I-T)\|=O(n^{-1/2})$ as $n\to\infty$.  This in turn implies, by a theorem of Dungey~\cite[Theorem~1.2]{Dun08}, that $\|T(t)A\|=O(t^{-1/2})$ as $t\to\infty$, where $(T(t))_{t\ge0}$ is the $C_0$-semigroup generated by $A$. The decay rate for $\dot{z}(t)$ as $t\to\infty$ now follows.

\subsection{Applications to Markov processes}
We end this section  by mentioning that the Katznelson-Tzafriri theorem and its variants have had numerous applications in the study of Markov processes and their long-time asymptotic behaviour. Indeed, as mentioned in the introduction, the main motivation behind Katznelson and Tzafriri's original paper~\cite{KT86} was to isolate an interesting and purely operator-theoretic fact from a complicated argument given by Zaharopol~\cite{Zah86} (building on earlier work by Ornstein and Sucheston~\cite{OrnSuc70} and by Foguel~\cite{Fog76}) to prove a certain ``zero-two law'' for Markov processes. Markov processes have remained a fruitful area of application for results of Katznelson-Tzafriri type, and it would be beyond the scope of this survey to provide anything resembling a complete overview. We mention, however, that much of Dungey's influential work on rates of decay in the Katznelson-Tzafriri theorem was directly motivated by his work on random walks on groups (see \cite{Dun07a, Dun07b, Dun07c, Dun08b, Dun08c}), and we refer the interested reader to  \cite{CoCuLi14, CohLin16, CouSC90} for further results relating to the Katznelson-Tzafriri theorem in the context of Markov processes and positive contractions on $L^p$-spaces.

\section{Other operator semigroups} \label{sect5}

In this section, we briefly indicate how results in Sections \ref{sect2} and \ref{sect3} can be formulated for other operator semigroups.

\subsection{$C_0$-semigroups} \label{sect5.1}
Some of the theorems of Katznelson-Tzafriri type have analogues for $C_0$-semigroups.   Let $(T(t))_{t\ge0}$ be a bounded $C_0$-semigroup on a Banach space $X$, with generator $-A$.   There are some conditions which ensure that $\lim_{t\to\infty} \|T(t)f(A)\| = 0$ for appropriate functions $f$.   In particular, if $f$ is the Laplace transform of a function $g \in L^1(\R_+)$, then $f(A)$ may be defined by the Hille-Phillips functional calculus as 
\[
f(A)x = \int_0^\infty g(t) T(t)x \,dt, \qquad x \in X.
\]
A direct analogue of Theorem \ref{KTthm} for this class of functions was proved in \cite{ESZ2} and \cite{Vu92}, independently by different methods.  See also \cite[Theorem 5.7.4]{ABHN}, \cite[Theorem 6.14]{BCT} and \cite{Sei15a}.   

The analogue of Corollary \ref{KTcor} is the result that $\lim_{t\to\infty} \|T(t)A^{-1}\| = 0$ if $\sigma(A) \cap i\R$ is empty, which was first noted in \cite{Bat90}.   The analogues of Theorems \ref{thm:mlog} and \ref{thm:m_inv} on rates of convergence were proved in \cite{BD08} and \cite{RSS} respectively, before the results for the discrete case.

The Banach algebra $\B(\D)$ has appeared in the literature in various contexts.   There is a corresponding space $\B(\C_+)$ of analytic Besov functions on the right half-plane $\C_+$.   It had appeared in the literature only very occasionally until it was studied in detail in \cite{BGT} and \cite{BGT2}.  A  $\B$-calculus  for (unbounded) operators $A$ which are negative generators of bounded $C_0$-semigroups was also studied in those papers.   The theory of the algebras and of their functional calculi is more complicated in the case of $\C_+$, but analogous results have been obtained in most cases. The results in Section \ref{sect2} of this paper are similar to those obtained in \cite[Section 4]{BGT} and \cite[Section 4]{BGT2},
 with the exception of Theorem \ref{BSthm}.    A shorter version of the papers is included in \cite{BGTsum}.

It is possible to formulate to a statement which, if true, would be a continuous-parameter version of Theorem \ref{BSthm}, but we have not yet been able to determine whether or not it is true.

\subsection {Several commuting operators}
\def\boldT{{\mathbf T}}

Let $\boldT := (T_1,\dots,T_m)$ be an $m$-tuple of commuting bounded linear operators on a Banach space $X$.   It is of some interest to consider versions of the Katznelson-Tzafriri theorem for such an $m$-tuple.   Indeed, Katznelson and Tzafriri raised this in \cite[Section 4]{KT86}.   For such a theorem, we need to have a notion of the unitary joint spectrum for ${\bf T}$.   Although there are many different notions of the joint spectrum of $\boldT$, they coincide in the case of the unitary joint spectrum for power-bounded operators.   
The {\it unitary joint spectrum} of $\boldT$ is the set of all $(\l_1,\dots,\l_m) \in \T^m$ for which there exists a sequence $(x_k)_{k\ge1}$ in $X$ such that
\[
\lim_{k\to\infty} \|T_j x_k - \l_j x_k\| \to 0, \qquad k\to\infty, \; j=1,\dots,m.
\]

In \cite[Theorem 3.8]{KT86} and \cite[Corollary 2.6]{AR89}, it was proved that
\[
\lim_{n \to \infty} \|(T_1\dots T_m)^n f(\boldT)\| = 0,
\]
if $f$ is in the $m$-dimensional Wiener algebra  $W(\D^m)$ and $f$ is of spectral synthesis in $W(\T^m)$ with respect to  $\sigma_u(\boldT)$.     The paper \cite{AR89} also contained some weighted versions of the results.   Bercovici's paper \cite{Ber90} included a version for algebras which are $m$-dimensional versions of the algebras considered in \cite{ESZ1}.

In \cite[Theorem 4.1]{Zar13}, Zarrabi considered the case of $m$ commuting contractions on a Hilbert  space, and functions $f \in W^+(\D^m)$ which do not necessarily vanish on $\sigma_u(\boldT)$.   He extended Theorem \ref{Zarthm} by showing that
\begin{equation} \label{zarbd}
\lim_{n \to \infty} \|(T_1\dots T_m)^n f(\boldT)\| = \sup \{|f(\l_1,\dots,\l_m)|: (\l_1,\dots,\l_m) \in \sigma_u(\boldT) \}.
\end{equation}

\subsection{General semigroups}  
 In \cite[p.365]{Ber90}, Bercovici wrote
\begin{quote} {\it  Since the consideration of $k$-tuples leads to cumbersome notation, we will work with a somewhat more abstract setting.}
\end{quote}
For Bercovici, the more abstract setting was the case of a semigroup $S$ contained in a discrete abelian group $G$, and a semigroup homomorphism from $S$ to $L(X)$.   For example, the case of several commuting operators corresponds to $G = \Z^m$ and $S =\Z_+^m$.

One may go further, and consider a locally compact abelian group $G$, and a measurable semigroup $S$ contained in $G$, with non-empty interior, and with $G = S - S$.   In general, one is considering representations of $S$ on a Banach space $X$, that is, semigroup homomorphisms $T : S \to \L(X)$ which are continuous in the strong operator topology.    When $G = \R$ and $S = \R_+$, the representations of $S$ are precisely the $C_0$-semigroups.   By considering the Haar measure of $G$ on the subgroup $S$, one may consider functions $g \in L^1(S)$.  For a bounded representation $T$ on a Banach space $X$, bounded operators on $X$ are defined by
\[
\widehat{g}(T)x = \int_S g(t) T(t) x \,dt, \qquad x \in X.
\]
The unitary spectrum $\operatorname{Sp}_u(T,S)$ of $T$ with respect to $S$ can be defined as a subset of the dual group of $G$, and the notion of spectral synthesis in $L^1(G)$ can also be defined.   The following generalisation of Theorem \ref{KTthm} was proved in \cite{BV92}.

\begin{thm} \label{Bvthm}
Let $G$ and $S$ be as above, and $T$ be a bounded representation of $S$ on a Banach space.  Let $g \in L^1(S)$, and assume that $g$ is of spectral synthesis in $L^1(G)$ with respect to $\operatorname{Sp}_u(T,S)$.   Then $\lim_S \|T(t) \widehat{g}(T)\|=0$.
\end{thm}

Zarrabi \cite[Theorem 2.6]{Zar13} considered representations of semigroups in this class by contractions on Hilbert spaces, obtaining a more general version of \eqref{zarbd} and a similar formula for $C_0$-semigroups.    The paper \cite{Sei15c} extends Theorem \ref{Lekthm} to bounded representations of these semigroups on Hilbert spaces, and it also extends Zarrabi's result \eqref{zarbd} for contractive representations on Hilbert spaces.

\subsection*{Conflict of interest statement}
 On behalf of all authors, the corresponding author states that there is no conflict of interest.
 
\subsection*{Data availability statement}
No datasets were generated or analysed during the current study.

\end{document}